\documentclass{article}

\usepackage[utf8]{inputenc}
\usepackage{amsmath, amsfonts, amsthm, mathtools}
\usepackage{babel}
\usepackage{hyperref}
\usepackage[hyphenbreaks]{breakurl}
\usepackage{authblk}

\usepackage{tikz}
\usetikzlibrary {arrows.meta, positioning, math}

\newtheorem{theorem}{Theorem}[section]
\newtheorem{lemma}[theorem]{Lemma}

\newcommand{\tikzmark}[1]{\tikz[overlay, remember picture] \coordinate (#1);}

\usepackage{biblatex}

\title{Fast matrix representation for Clifford algebras}
\author{Gleb Rumyantsev}
\date{2024}

\affil{\href{mailto:h1m3@proton.me}{h1m3@proton.me}}

\begin{document}

\maketitle

\begin{abstract}
In this paper, we present two fast matrix representation algorithms based on the recursive decomposition of multivectors into specific right and left ideals. We also examine the relation between these two representations. Furthermore, we derive the explicit forms of the fundamental (anti)automorphisms of these Clifford algebra representations, and the efficient methods to compute them. The algorithms have been implemented in Rust and are available as the Cargo crate \textit{clifft} on Github, released under the MIT license.
\end{abstract}

\section{Introduction}

The term \emph{fast matrix representation} in this paper refers to the algorithm that implements homomorphism function $F: {Cl(p,q) \rightarrow M}$ where $M$ is some matrix algebra. The algorithm input is an array of coefficients for every k-blade of a multivector, and its output is the matrix.

This is extremely useful because the element-wise multiplication of the coefficients can only be done with $O(N^2)$ complexity, where $N$ is the number of coefficients. The matrices, however, can be multiplied much faster. So, fast matrix representation plays the same role for Clifford Algebras as Fast Fourier Transform does for polynomials.

The right-spinor representation discussed in this paper is precisely the generalized FFT proposed by P. Leopardi \autocite{leopardiGeneralizedFFTClifford2005}. In addition, we show that there is a similar algorithm to represent the algebra in the left-spinor basis, and explore the relation between the two representations. Futhermore, we derive the explicit form of the fundamental (anti)automorphisms of Clifford algebras in matrix form, and the efficient ways to compute them. The algorithms are implemented in Rust and are available as the Cargo crate \textit{clifft} on Github \autocite{rumyantsevClifft2024}.

We mostly focus on $Cl(n,n)$ algebras over field $\mathbb{K}$ of characteristic $\ne 2$. These algebras are the easiest ones to work with. They are also particularly useful, since they naturally occur as the Clifford algebras over $V \oplus V^*$ \autocite{gualtieriGeneralizedComplexGeometry2004}. Such algebras have even been applied for describing the Standard model symmetries from purely geometrical standpoint \autocite{stoicaStandardModelAlgebra2018}.

Other Clifford algebras can be simply viewed as (or being isomorphic to) sub-algebras of $Cl(n,n)$ for the appropriate $n$. Explicit ways to construct their representations from the $Cl(n,n)$ representations are shown in \autocite{sMatrixRepresentationsClifford2023}. More general properties of such algebras with a split bilinear form and their relations to other algebras can be found in \autocite{meinrenkenCliffordAlgebrasLie2013}.

This paper follows the notation conventions below:
\begin{itemize}
	\item
	Basis vectors that are squared to $-1$ are written with tilde: $\tilde{e}_i$, and those that square to $1$ as $e_i$.
	\item
	In the $Cl(n,n)$ algebra, basis vectors are enumerated from $e_{0}$ to $e_{n-1}$ and $\tilde{e}_0$ to $\tilde{e}_{n-1}$.
	\item
	$\alpha$ denotes the parity automorphism, defined by ${\alpha(e_i) = -e_i}$, ${\alpha(\tilde{e}_i) = -\tilde{e}_i}$ for all basis vectors.
	\item
	$\beta$ denotes the parity flip of all negative-valued basis vectors: ${\beta(e_i) = e_i}$, ${\beta(\tilde{e}_i) = -\tilde{e}_i}$.
	\item The star operator $A^*$ refers to the reversal of the multivector $A$.
\end{itemize}

\section{The fast matrix representation}

First, we go through some useful properties of the algebras.

\begin{lemma}
\label{thm:alphacommute}
	For every 1-vector $e$ and multivector $A$ such that no blades of $A$ contain $e$,
	\[ e A = \alpha(A)e \]
\end{lemma}
\begin{proof}
	Each blade in $A$ anticommutes with $e$ if its grade is odd, and commutes with $e$ if its grade is even.
\end{proof}

Consider $Cl(p+1, q+1)$ over a field $\mathbb{K}$ with characteristic $\ne 2$. In this algebra, there exists an orthogonal pair of projectors
\begin{equation}
\begin{array}{l}
	P_+ = \frac{1}{2}(1 + \tilde{e} e) \\
	P_- = \frac{1}{2}(1 - \tilde{e} e) \\
\end{array}
\end{equation}

where $e,\tilde{e}$ are basis vectors, and $e^2 = 1, \tilde{e}^2 = -1$.
The projectors $P_\pm$ commute with all k-vectors orthogonal to $e$ and $\tilde{e}$. Relations between the projectors and $e,\tilde{e}$ are as follows:
\begin{equation}
\begin{array}{l}
	P_+ e = P_+ \tilde{e} \\
	P_- e = - P_- \tilde{e}  \\
	e P_\pm = P_\mp e \\
\end{array}
\label{eqn:relations}
\end{equation}

\subsection{Right-spinor basis}

Let us decompose multivectors $a,b \in Cl(p+1, q+1)$ and $c = ab$ as
\begin{equation}
\begin{array}{l}
	a = P_+ a_0^+ + P_+ e a_1^+ + P_- a_0^- + P_- e a_1^- \\
	b = P_+ b_0^+ + P_+ e b_1^+ + P_- b_0^- + P_- e b_1^- \\
	c = P_+ c_0^+ + P_+ e c_1^+ + P_- c_0^- + P_- e c_1^- \\
\end{array}
\label{eqn:rdec}
\end{equation}	
where $a_i^\pm, b_i^\pm, c_i^\pm \in Cl(p,q)$ sub-algebra that doesn't contain \(e\) and \(\tilde{e}\).
Consider a product $c = a b$ under such decomposition, then (using lemma \ref{thm:alphacommute}):
\begin{equation}
\begin{array}{l}
	P_+ c_0^+ = {P_+ a_0^+ P_+ b_0^+} + {P_+ e a_1^+ P_- e b_1^-}
	= P_+ \left({a_0^+ b_0^+} + {\alpha(a_1^+) b_1^-} \right) \\
	P_- c_0^- = {P_- a_0^- P_- b_0^-} + {P_- e a_1^- P_+ e b_1^+}
	= P_-\left(a_0^- b_0^- + \alpha(a_1^-) b_1^+ \right)\\
	P_+ e c_1^+ = {P_+ e a_1^+ P_- b_0^-} + {P_+ a_0^+ P_+ e b_1^+} 
	= P_+ e \left({a_1^+ b_0^-} + {\alpha(a_0^+) b_1^+}\right)\\
	P_- e c_1^- = {P_- e a_1^- P_+ b_0^+} + {P_- a_0^- P_- e b_1^-} 
	= P_- e \left({a_1^- b_0^+} + {\alpha(a_0^-) b_1^-}\right)\\
\end{array}
\end{equation}

This contains a lot of seemingly independent variables. However, by applying the parity automorphism to 2 of the equations, we obtain the following equations:
\begin{equation}
\begin{array}{l}
	c_0^+ = {a_0^+ b_0^+} + {\alpha(a_1^+) b_1^-} \\
	\alpha (c_0^-) = \alpha(a_0^-) \alpha(b_0^-) + a_1^- \alpha(b_1^+)\\
	\alpha (c_1^+) = \alpha(a_1^+) \alpha(b_0^-) + a_0^+ \alpha(b_1^+)\\
	c_1^- = {a_1^- b_0^+} + {\alpha(a_0^-) b_1^-} \\
\end{array}
\end{equation}
	
This can be rewritten in matrix form:
\begin{equation}
\begin{bmatrix}
	c_0^+ & \alpha(c_1^+) \\
	c_1^- & \alpha(c_0^-) \\
\end{bmatrix}
=
\begin{bmatrix}
	a_0^+ & \alpha(a_1^+) \\
	a_1^- & \alpha(a_0^-) \\
\end{bmatrix}
\begin{bmatrix}
	b_0^+ & \alpha(b_1^+) \\
	b_1^- & \alpha(b_0^-) \\
\end{bmatrix}
\end{equation}
thus providing us with a way to build an isomorphism $Cl(p+1, q+1) \cong M_2(Cl(p, q))$.

When a multivector is provided by an array of coefficients, it can be naturally viewed as consisting of the 4 parts, depending on the inclusion of $e, \tilde{e}$ in the blades:
\begin{equation}
a = A_{00} + e A_{01} + \tilde{e} A_{10} + \tilde{e}e A_{11}
\label{eqn:rblades}
\end{equation}

Its projections by $P_+, P_-$ are (using the relations \eqref{eqn:relations}):
\begin{equation}
\begin{array}{l}
	P_+a = P_+A_{00} + P_+ e A_{01} + P_+ e A_{10} + P_+ A_{11} \\
	P_-a = P_+A_{00} + P_- e A_{01} - P_- e A_{10} - P_- A_{11} \\
\end{array}
\end{equation}

Thus, the two decompositions are related as follows:
\begin{equation}
\begin{matrix*}[l]
	a_0^+ = A_{00} + A_{11}, & a_1^+ = A_{01} + A_{10}, \\
	a_0^- = A_{00} - A_{11}, & a_1^- = A_{01} - A_{10}. \\
\end{matrix*}
\end{equation}

So, the matrix representation of $a$ is
\begin{equation}
\begin{pmatrix}
	A_{00} + A_{11} & \alpha(A_{01} + A_{10}) \\
	A_{01} - A_{10} & \alpha(A_{00} - A_{11})
\end{pmatrix} {}
\end{equation}

The inverse relation is
\begin{equation}
\begin{array}{ll}
	A_{00} = \frac{1}{2}(a_0^+ + a_0^-) &
	A_{01} = \frac{1}{2}(a_1^+ + a_1^-) \\
	A_{10} = \frac{1}{2}(a_1^+ - a_1^-) &
	A_{11} = \frac{1}{2}(a_0^+ - a_0^-) \\
\end{array}
\end{equation}

With that, the original array of coefficients can be restored from the matrix.

\subsection{Fast matrix representation for Cl(n,n)}

\begin{itemize}
	\item 
	For $Cl(0,0) \cong \mathbb{K}$ matrix representation is just the scalar: $F(a) = (a)$.
	\item
	For $n \ge 1$, $a \in Cl(n,n)$ we get the recurrent formula:
	\begin{equation}
		\begin{split}
		F_{n}(A_{00} + e_{n-1} A_{01} + \tilde{e}_{n-1} A_{10} + \tilde{e}_{n-1} e_{n-1} A_{11}) = \\
		\hfill =
		\begin{bmatrix}
			F_{n-1}(A_{00} + A_{11}) & F_{n-1}\left(\alpha(A_{01} + A_{10})\right) \\
			F_{n-1}(A_{01} - A_{10}) & F_{n-1}\left(\alpha(A_{00} - A_{11})\right)
		\end{bmatrix}
		\end{split}
	\end{equation}
\end{itemize}

This can be nicely depicted as diagram \ref{fig:clifft}.

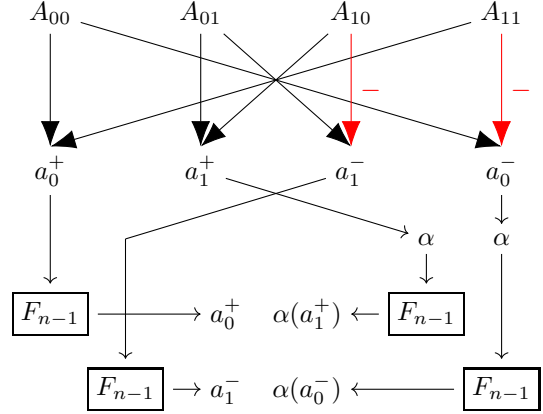
\begin{figure}[h]
\centering
\begin{tikzpicture}
	\node (A00) at (0,0) {$A_{00}$};
	\node (A01) at (2,0)  {$A_{01}$};
	\node (A10) at (4,0) {$A_{10}$};
	\node (A11) at (6,0) {$A_{11}$};
	
	\node (a) [below=1.5 of A00] {$a_0^+$};
	\node (b) [below=1.5 of A01] {$a_1^+$};
	\node (c) [below=1.5 of A10] {$a_1^-$};
	\node (d) [below=1.5 of A11] {$a_0^-$};
	
	\node (dp) at (6, -3) {$\alpha$};
	\node (bp) at (5, -3) {$\alpha$};
	\coordinate (cp) at (1, -3);
	
	\draw [-{Latex[scale=2]}] (A00) -- (a.north);
	\draw [-{Latex[scale=2]}] (A11) -- (a.north);
	\draw [-{Latex[scale=2]}] (A00) -- (d.north);
	\draw [-{Latex[scale=2]}, red] (A11) -- (d.north) node [midway, right] {$-$};
	
	\draw [-{Latex[scale=2]}] (A01) -- (b.north);
	\draw [-{Latex[scale=2]}] (A10) -- (b.north);
	\draw [-{Latex[scale=2]}] (A01) -- (c.north);
	\draw [-{Latex[scale=2]}, red] (A10) -- (c.north) node [midway, right] {$-$};
	
	\node (m00) at (2, -4) [right] {$a_0^+$};
	\node (m01) at (4, -4) [left]  {$\alpha(a_1^+)$};
	\node (m10) at (2, -5) [right] {$a_1^-$};
	\node (m11) at (4, -5) [left]  {$\alpha(a_0^-)$};
	
	\node (f00) at (0, -4) {\fbox{$F_{n-1}$}};
	\node (f10) at (1, -5) {\fbox{$F_{n-1}$}};
	\node (f01) at (5, -4) {\fbox{$F_{n-1}$}};
	\node (f11) at (6, -5) {\fbox{$F_{n-1}$}};
	
	\draw [->] (d) -- (dp);
	\draw [->] (b) -- (bp);
	\draw [-] (c) -- (cp);
	
	\draw [->] (a) -- (f00);
	\draw [->] (cp) -- (f10);
	\draw [->] (bp) -- (f01);
	\draw [->] (dp) -- (f11);	
	
	\draw [->] (f00) -- (m00);
	\draw [->] (f10) -- (m10);
	\draw [->] (f01) -- (m01);
	\draw [->] (f11) -- (m11);	
\end{tikzpicture}
\caption{Fast matrix representation diagram}
\label{fig:clifft}
\end{figure}

The structure of this algorithm closely resembles the Fast Walsh-Hadamard transform \autocite{shanksComputationFastWalshFourier1969}, but it operates on quarters of the transformed object instead of halves.

The inverse transform is shown on the diagram \ref{fig:iclifft}.

\begin{figure}[h]
\centering
\begin{tikzpicture}	
	\node (A00) at (0,0) {$2A_{00}$};
	\node (A01) at (2,0)  {$2A_{01}$};
	\node (A10) at (4,0) {$2A_{10}$};
	\node (A11) at (6,0) {$2A_{11}$};
		
	\node (AA00) at (0,-1) {$A_{00}$};
	\node (AA01) at (2,-1)  {$A_{01}$};
	\node (AA10) at (4,-1) {$A_{10}$};
	\node (AA11) at (6,-1) {$A_{11}$};
	
	\draw [->] (A00) -- (AA00) node [midway, right] {$/2$};
	\draw [->] (A01) -- (AA01) node [midway, right] {$/2$};
	\draw [->] (A10) -- (AA10) node [midway, right] {$/2$};
	\draw [->] (A11) -- (AA11) node [midway, right] {$/2$};
	
	\node (a) [above=1.5 of A00] {$a_0^+$};
	\node (b) [above=1.5 of A01] {$a_1^+$};
	\node (c) [above=1.5 of A10] {$a_1^-$};
	\node (d) [above=1.5 of A11] {$a_0^-$};
	
	\node (dp) at (6, 3) {$\alpha$};
	\node (bp) at (5, 3) {$\alpha$};
	\coordinate (cp) at (1, 3);
	
	\draw [{Latex[scale=2]}-] (A00.north) -- (a);
	\draw [{Latex[scale=2]}-] (A11.north) -- (a);
	\draw [{Latex[scale=2]}-] (A00.north) -- (d);
	\draw [{Latex[scale=2]}-, red] (A11.north) -- (d) node [midway, right] {$-$};
	
	\draw [{Latex[scale=2]}-] (A01.north) -- (b);
	\draw [{Latex[scale=2]}-] (A10.north) -- (b);
	\draw [{Latex[scale=2]}-] (A01.north) -- (c);
	\draw [{Latex[scale=2]}-, red] (A10.north) -- (c) node [midway, right] {$-$};
	
	\node (m00) at (2, 5) [right] {$a_0^+$};
	\node (m01) at (4, 5) [left]  {$\alpha(a_1^+)$};
	\node (m10) at (2, 4) [right] {$a_1^-$};
	\node (m11) at (4, 4) [left]  {$\alpha(a_0^-)$};
	
	\node (f00) at (0, 5) {\fbox{$F^{-1}_{n-1}$}};
	\node (f10) at (1, 4) {\fbox{$F^{-1}_{n-1}$}};
	\node (f01) at (5, 5) {\fbox{$F^{-1}_{n-1}$}};
	\node (f11) at (6, 4) {\fbox{$F^{-1}_{n-1}$}};
	
	\draw [<-] (d) -- (dp);
	\draw [<-] (b) -- (bp);
	\draw [<-] (c) -- (cp);
	
	\draw [<-] (a) -- (f00);
	\draw [-] (cp) -- (f10);
	\draw [<-] (bp) -- (f01);
	\draw [<-] (dp) -- (f11);	
	
	\draw [<-] (f00) -- (m00);
	\draw [<-] (f10) -- (m10);
	\draw [<-] (f01) -- (m01);
	\draw [<-] (f11) -- (m11);	
\end{tikzpicture}
\caption{Fast inverse matrix representation diagram}
\label{fig:iclifft}
\end{figure}
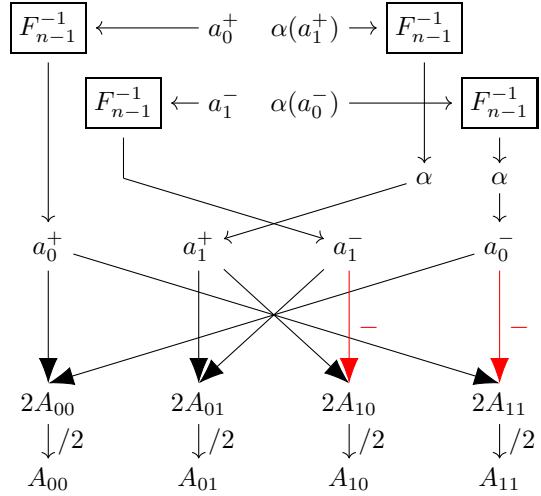

\begin{lemma}
\label{thm:symmetry}
\(\forall n, 0 \le k < n\): \(F_n(e_k)\) is symmetric, \(F_n(\tilde{e}_k)\) is antisymmetric.
\end{lemma}
\begin{proof} 
\(
F_k(e_{k-1}) =
\begin{bmatrix}
	F_{k-1}(0) & F_{k-1}(1) \\
	F_{k-1}(1) & F_{k-1}(0) \\
\end{bmatrix}
\) is symmetric for all $k$. \\
For \({0 \le j < k-1}\), \(
F_k(e_j) =
\begin{bmatrix}
	F_{k-1}(e_j) & F_{k-1}(0) \\
	F_{k-1}(0) & F_{k-1}(e_j) \\
\end{bmatrix}
\)
each $F_{k-1}(e_j)$ is symmetric by induction. 

Similarly, 
\(F_k(\tilde{e}_{k-1}) =
\begin{bmatrix}
	F_{k-1}(0) & F_{k-1}(1) \\
	F_{k-1}(-1) & F_{k-1}(0) \\
\end{bmatrix}
\) is antisymmetric for all $k$, and 
for \({0 \le j < k-1}\), \(
F_k(\tilde{e}_j) =
\begin{bmatrix}
	F_{k-1}(\tilde{e}_j) & F_{k-1}(0) \\
	F_{k-1}(0) & F_{k-1}(\tilde{e}_j) \\
\end{bmatrix}
\)
each $F_{k-1}(\tilde{e}_j)$ is antisymmetric by induction. 
\end{proof}

\subsection{Algorithm complexity}

The representation of $Cl(n+1,n+1)$ involves 4 coefficient-wise additions ($2^{2n} = 4^n$ additions each), 2 parity automorphisms in $Cl(n,n)$, each involving $2^{2n}$ operations with $O(1)$ complexity, and 4 representation operation for $Cl(n,n)$ multivectors.

Assuming that both the addition of 2 numbers and the sign flip of a single number for the parity automorphism have complexity of 1, the complexity of the representation algorithm $C_{n+1}$ can be expressed as the recurrence formula:
\begin{equation}
C_{n+1} = 6 \cdot 4^n + 4 C_{n}
\end{equation}

Then,
\begin{equation}
C_{n+1} = 6 \cdot 4^n + 4(6 \cdot 4^{n-1} + 4(6 \cdot 4^{n-2} + 4 (...))) \sim 6 \cdot n \cdot 4^n
\end{equation}

Or, simply
\begin{equation}
C_n = O(n \cdot 4^n) = O(\log(N) \cdot N)
\end{equation}
where $N = 4^n$ is the number of coefficients in the $Cl(n,n)$ multivector.

The inverse transform algorithm consists of exactly the same operations, and, therefore, has the same complexity.

\subsection{Left-spinor basis}

The exact same reasoning can be applied for the decomposition into the left ideals:
\begin{equation}
a = \bar{a}_0^+ P_+ + \bar{a}_1^+ e P_+ + \bar{a}_0^- P_- + \bar{a}_1^- e P_-
\label{eqn:ldec}
\end{equation}

Using relations \eqref{eqn:relations},
\begin{equation}
a = a_0^+ P_+ + \alpha(a_1^-) e P_+ + a_0^- P_- + \alpha(a_1^+) e P_-
\end{equation}

Thus,
\begin{equation}
\begin{bmatrix}
	\bar{a}_0^+ & \bar{a}_1^- \\
	\alpha(\bar{a}_1^+) & \alpha(\bar{a}_0^-) \\
\end{bmatrix}
=
\begin{bmatrix}
	a_0^+ & \alpha(a_1^+) \\
	a_1^- & \alpha(a_0^-) \\
\end{bmatrix}
\label{eqn:lrcorrespondence}
\end{equation}
can be used as an alternative way to represent the same multivector.

The left projections of $a$ are:
\begin{equation}
\begin{array}{l}
	a P_+ = \bar{A}_{00} P_+ + \bar{A}_{01} e P_+ + \bar{A}_{10} e P_+ + \bar{A}_{11} P_+\\
	a P_- = \bar{A}_{00} P_- + \bar{A}_{01} e P_- - \bar{A}_{10} e P_- - \bar{A}_{11} P_- \\
\end{array}
\end{equation}

From there,
\begin{equation}
\begin{array}{l}
	a P_+ = (\bar{A}_{00} + \bar{A}_{11}) P_+ + (\bar{A}_{01} + \bar{A}_{10}) e P_+ \\
	a P_- = (\bar{A}_{00} - \bar{A}_{11}) P_- + (\bar{A}_{01} - \bar{A}_{10}) e P_-
\end{array}
\end{equation}

or,
\begin{equation}
\begin{matrix*}[l]
	\bar{a}_0^+ = \bar{A}_{00} + \bar{A}_{11}, & \bar{a}_1^+ = \bar{A}_{01} + \bar{A}_{10}, \\
	\bar{a}_0^- = \bar{A}_{00} - \bar{A}_{11}, & \bar{a}_1^- = \bar{A}_{01} - \bar{A}_{10}. \\
\end{matrix*}
\end{equation}

Therefore, the explicit recursive formula for the algorithm in the left-spinor basis is
\begin{equation}
\begin{split}
F_{n}(\bar{A}_{00} + \bar{A}_{01} e_{n-1} + \bar{A}_{10} \tilde{e}_{n-1} + \bar{A}_{11} \tilde{e}_{n-1} e_{n-1}) =
\\
\hfill =
\begin{bmatrix}
	F_{n-1}(\bar{A}_{00} + \bar{A}_{11}) & F_{n-1}(\bar{A}_{01} - \bar{A}_{10}) \\
	F_{n-1}\left(\alpha(\bar{A}_{01} + \bar{A}_{10})\right) &
	F_{n-1}\left(\alpha(\bar{A}_{00} - \bar{A}_{11})\right) 
\end{bmatrix}
\end{split}
\end{equation}

For restoring the original coefficients,
\begin{equation}
\begin{matrix*}[l]
	\bar{A}_{00} = \frac{1}{2}(\bar{a}_0^+ + \bar{a}_0^-) &
	\bar{A}_{01} = \frac{1}{2}(\bar{a}_1^+ + \bar{a}_1^-) \\
	\bar{A}_{10} = \frac{1}{2}(\bar{a}_1^+ - \bar{a}_1^-) &
	\bar{A}_{11} = \frac{1}{2}(\bar{a}_0^+ - \bar{a}_0^-) \\
\end{matrix*}
\end{equation}

The choice between left and right basis ultimately depends on convenience.

\subsection{Structure of the matrix}

Let's denote projectors for the minimal ideals as
\begin{equation}
P_\sigma = \prod_i \frac{1}{2}(1 + (-1)^{\sigma_i} \tilde{e}_i e_i)
\end{equation}
where $\sigma$ is a bit string of length $n$. Since bits of $\sigma$ are responsible for the sign, we will interchangeably use "+" for "0"  and "-" for "1".

The recursive \textbf{left-spinor representation} of $a\in Cl(n,n)$ then can be expressed as
\begin{equation}
a = \sum_{\sigma} \left( \sum_{\rho} \left(a^{\sigma}_\rho e^{\rho} \right) \right) P_\sigma
\end{equation}

Here, indices $\sigma$ and $\rho$ go through all the binary strings of length $n$, and $e^\rho$ is a shorthand for the blade of the form $e^{\rho_0}e^{\rho_1} \cdots e^{\rho_{n-1}}$.

Each $a^{\sigma}_{\rho}$ corresponds (up to a sign) to the matrix entry at column $\rho$ and row $\sigma \oplus \rho$, where $\oplus$ denotes bitwise xor. The example for $Cl(2,2)$:
\begin{equation}
\bordermatrix{
	\sigma &\hfill ++ &\hfill +- &\hfill -+ &\hfill -- \cr
	&\hfill a_{00}^{++} &\hfill a_{01}^{+-} &\hfill a_{10}^{-+} &\hfill a_{11}^{--} \tikzmark{up} \cr
	&\hfill a_{01}^{++} &\hfill a_{00}^{+-} &\hfill a_{11}^{-+} &\hfill a_{10}^{--} \cr
	&\hfill a_{10}^{++} &\hfill - a_{11}^{+-} &\hfill a_{00}^{-+} &\hfill - a_{01}^{--} \cr
	&\hfill - a_{11}^{++} &\hfill a_{10}^{+-} &\hfill - a_{01}^{-+} &\hfill a_{00}^{--} \tikzmark{down} \cr
}
\end{equation}

\tikz[overlay, remember picture] {
	\draw [->] ([xshift=2em,yshift=1.5ex]up) -- ([xshift=2em,yshift=-0.5ex]down) node [midway, right] {$\sigma \oplus \rho$};
}

The \textbf{right-spinor representation} behaves in exactly the same way, except the $\sigma$ becomes the row index, and $\rho \oplus \sigma$ enumerates the columns.

The weight of the index $\rho$ matches the dimensionality of the blade $e^\rho$. This can be used to assign a grade to every entry of the matrix. The parity of this grade is preserved under multiplication by even multivectors. This property is particularly useful for constructing half-spin representations of the Spin group.

\subsection{Complex case}

Everything that applies to the complex algebra $\mathbb{C}l(n,n)$ is valid for $\mathbb{C}l(2n)$ as well, under the following isomorphism $c: \mathbb{C}l(2n) \rightarrow \mathbb{C}l(n, n)$ between the two:
\begin{equation}
\begin{array}{c|c}
	\mathbb{C}l(2n) & \mathbb{C}l(n,n) \\
	\hline
	e'_{2k} & e_k\\
	i \cdot e'_{2k+1} & \tilde{e}_k \\
\end{array}
\end{equation}

\newcommand{\FC}{F_{\mathbb{C}}}
Fast complex matrix representation $\FC$ for $\mathbb{C}l(2n)$ can therefore be constructed as
\begin{equation}
	\FC(x) = F(c(x))
\end{equation}

Any other real or complex $Cl(p,q)$ can be viewed as being isomorphic to a subalgebra of $\mathbb{C}l(2n)$, where $n = \left\lceil\frac{p+q}{2}\right\rceil$. As such, $\FC$ provides the most general way to efficiently represent Clifford algebras over $\mathbb{R}$ and $\mathbb{C}$.

\(\FC\) is implemented in the \textit{clifft} crate \autocite{rumyantsevClifft2024} as the function \texttt{clifft}, and the inverse \(\FC^{-1}\) as \texttt{iclifft}.

\begin{lemma}
	\(\forall n, 0 \le k < 2n\): \(F_{\mathbb{C}}(e_k)\) is Hermitian.
\end{lemma}
\begin{proof}
	For odd $k$, \(\FC(e'_k) = F(e_{k/2}) \) is symmetric (by lemma \ref{thm:symmetry}) and real;
	for even $k$ \(\FC(e'_k) = F(i \tilde{e}_{(k-1)/2}) \) is antisymmetric and imaginary.
\end{proof}

\section{(Anti)automorphisms of the representations}

\subsection{Parity automorphism and flip of imaginary axes}
Projectors $\frac{1}{2}(1 \pm \tilde{e} e)$ are invariant under the parity transform. Using the decomposition \eqref{eqn:rdec},
\begin{equation}
\alpha(a) = P_+ \alpha(a_0^+) - P_+ e \alpha(a_1^+) + P_- \alpha(a_0^-) - P_- e \alpha(a_1^-)
\end{equation}

In the matrix form, this corresponds to
\begin{equation}
\alpha
\begin{bmatrix}
	a_0^+ & \alpha(a_1^+) \\
	a_1^- & \alpha(a_0^-) \\
\end{bmatrix}
=
\begin{bmatrix}
	\alpha(a_0^+) & -a_1^+ \\
	-\alpha(a_1^-) & a_0^- \\
\end{bmatrix}
\end{equation}

Thus, $\alpha$ only changes the signs of the entries. Inside the full $2^n \times 2^n$ representation of the $Cl(n,n)$ multivector, the sign flip occurs for the entries with index $\rho$ that have odd weight, as they correspond to blades of odd dimensionality. Therefore, every entry of $\alpha(a)$ can be calculated from the corresponding entry of $a$ as follows:
\begin{equation}
(\alpha(a))_{ij} = (-1)^{\left| i \oplus j \right|} a_{ij}
\end{equation}
where ${\left| i \oplus j \right|}$ is the weight of bitwise xor between the row and column indices. As such, the parity automorphism of the representation can be computed in $O(4^n) = O(N)$ time.

The $\alpha$ automorphism of the representations is implemented in the \textit{clifft} crate \autocite{rumyantsevClifft2024} as the function \texttt{parity\_flip}.

Under the reflection of imaginary axes $\beta$ that takes $\tilde{e}_i$ to $-\tilde{e}_i$, the projectors are exchanged:
\begin{equation}
\beta\left(\frac{1}{2}(1 \pm \tilde{e}_i e_i)\right) = \frac{1}{2}(1 \mp \tilde{e}_i e_i)
\end{equation}

Using the decomposition \eqref{eqn:rdec},
\begin{equation}
\beta(a) = P_- \beta(a_0^+) + P_- e \beta(a_1^+) + P_+ \beta(a_0^-) + P_+ e \beta(a_1^-)
\end{equation}

So the representation transforms under $\beta$ as follows:
\begin{equation}
\beta
\begin{bmatrix}
	a_0^+ & \alpha(a_1^+) \\
	a_1^- & \alpha(a_0^-) \\
\end{bmatrix}
=
\begin{bmatrix}
	\beta(a_0^-) & \alpha\beta(a_1^-) \\
	\beta(a_1^+) & \alpha\beta(a_0^+) \\
\end{bmatrix}
\end{equation}

This allows us to build another recursive algorithm for finding $\beta$ of a representation of a multivector:
\begin{enumerate}
	\item Exchange the quarters of the matrix;
	\item Compute $\alpha$ for each quarter;
	\item Recursively compute $\beta$ for each quarter.
\end{enumerate}
\begin{equation}
\begin{bmatrix}
	a_0^+ & \alpha(a_1^+) \\
	a_1^- & \alpha(a_0^-) \\
\end{bmatrix}
\overset{\text{1}}{\rightarrow}
\begin{bmatrix}
	\alpha(a_0^-) & a_1^- \\
	\alpha(a_1^+) & a_0^+ \\
\end{bmatrix}
\overset{\text{2}}{\rightarrow}
\begin{bmatrix}
	a_0^- & \alpha(a_1^-) \\
	a_1^+ & \alpha(a_0^+) \\
\end{bmatrix}
\overset{\text{3}}{\rightarrow}
\begin{bmatrix}
	\beta(a_0^-) & \alpha\beta(a_1^-) \\
	\beta(a_1^+) & \alpha\beta(a_0^+) \\
\end{bmatrix}
\end{equation}

The complexity of this algorithm can be calculated from the recurrence formula $C_{n+1} = O(4^n + 4 \cdot C_n) = O(n \cdot 4^n)$.

The $\beta$ automorphism of the representations is implemented in the \textit{clifft} crate \autocite{rumyantsevClifft2024} as the function \texttt{imaginary\_flip}.

\subsection{Transposition and reversal}

Consider transposition of the right-spinor representation:

\begin{equation}
\begin{bmatrix}
	a_0^+ & \alpha(a_1^+) \\
	a_1^- & \alpha(a_0^-) \\
\end{bmatrix}^T
= \begin{bmatrix}
	{a_0^+}^T & {a_1^-}^T \\
	\alpha({a_1^+}^T) & \alpha({a_0^-}^T) \\
\end{bmatrix}
\end{equation}
	
This is exactly the left-spinor representation \eqref{eqn:lrcorrespondence} of $a^T$. Therefore we can write
\begin{equation}
a^T = {a_0^+}^T P_+ + {a_1^+}^T e P_+ + {a_0^-}^T P_- + {a_1^-}^T e P_-
\end{equation}

Recursively, this applies reversal to each of the decomposition components, but not to the projectors themselves. This is equivalent to applying reversal and exchanging $\tilde{e_i} \rightarrow -\tilde{e_i}$ for every imaginary basis vector.

\begin{equation}
F(a)^T = F(\beta(a^*))
\end{equation}
As such, the representation of the reversal $F(a^*)$ can be computed from the representation $F(a)$ as

\begin{equation}
F(a^*) = \beta(F(a)^T)
\end{equation}
	
The complexity of reversal is bounded by the complexity of $\beta$, so for the proposed algorithm it's equal to $O(n \cdot 4^n)$.

Reversal of the representations is implemented in the \textit{clifft} crate \autocite{rumyantsevClifft2024} as the function \texttt{reversal}.

\printbibliography
\end{document}